\documentclass[11pt, a4paper]{article}
\usepackage[T1]{fontenc}
\usepackage[utf8]{inputenc}
\usepackage{geometry}
\usepackage{graphicx} 
\usepackage{amsthm}
\usepackage{amsmath}
\usepackage{thmtools}
\usepackage{hyperref}
\usepackage{amssymb}
\usepackage{xcolor}
\usepackage{comment}
\usepackage{xspace}
\usepackage{tikz}
\usepackage{complexity}
\usepackage{authblk}
\usepackage{mathtools}

\usepackage{setspace}

\usepackage[backend=biber, style=numeric, maxbibnames=99]{biblatex}
\renewbibmacro{in:}{\ifentrytype{article}{}
    {\printtext{\bibstring{in}\intitlepunct}}}
\DeclareFieldFormat[article]{pages}{#1}

\declaretheorem[name=Theorem, numberwithin=section]{theorem}
\declaretheorem[name=Lemma, sibling=theorem]{lemma}
\declaretheorem[name=Proposition, sibling=theorem]{proposition}

\declaretheorem[name=Corollary, sibling=theorem]{corollary}

\declaretheorem[name=Construction, sibling=theorem]{construction}

\declaretheorem[name=Question, sibling=theorem]{question}
\declaretheorem[name=Claim, numbered=no]{claim}

\declaretheorem[name=Observation, sibling=theorem]{observation}

\newtheorem{manualtheoreminner}{Theorem}
\newenvironment{manualtheorem}[1]{\IfBlankTF{#1}
    {\renewcommand{\themanualtheoreminner}{\unskip}}
    {\renewcommand\themanualtheoreminner{#1}}\manualtheoreminner}%
    {\endmanualtheoreminner}

\renewcommand{\bar}[1]{\overline{#1}}

\newcommand{\Ck}{\mathcal{C}_k(G)}
\newcommand{\KCk}{\mathcal{K}_k(G)}
\newcommand{\mcG}{\mathcal{G}}

\newcommand{\IkJ}{\mathcal{I}^k_{\rm J}}
\newcommand{\IkS}{\mathcal{I}^k_{\rm S}}
\newcommand{\IkAR}{\mathcal{I}^k_{\rm AR}}

\title{Determining a graph from its reconfiguration graph}
\author[1]{Ga\'etan Berthe}
\author[2]{Caroline Brosse}
\author[3]{Brian Hearn}
\author[3]{\linebreak Jan van den Heuvel}
\author[4]{Pierre Hoppenot}
\author[5]{Th\'eo Pierron}

\affil[1]{Universit\'e Clermont Auvergne, LIMOS, CNRS, Clermont-Ferrand, France\linebreak
\href{mailto:gaetan.berthe@limos.fr}{\tt gaetan.berthe@limos.fr}}
\affil[2]{Universit\'e d’Orl\'eans, INSA CVL, LIFO, UR 4022, Orl\'eans, France\linebreak
\href{mailto:caroline.brosse@univ-orleans.fr}{\tt caroline.brosse@univ-orleans.fr}}
\affil[3]{Department of Mathematics, London School of Economics \& Political Science, London, UK \linebreak
\href{mailto:b.hearn@lse.ac.uk}{\tt b.hearn@lse.ac.uk}, \href{mailto:j.van-den-heuvel@lse.ac.uk}{\tt j.van-den-heuvel@lse.ac.uk}}
\affil[4]{Universit\'e Grenoble Alpes, CNRS, Grenoble INP, G-SCOP, Grenoble, France\linebreak
\href{mailto:pierre.hoppenot@grenoble-inp.fr}{\tt pierre.hoppenot@grenoble-inp.fr}}
\affil[5]{Universit\'e Lyon, UCBL, CNRS, INSA Lyon, LIRIS, UMR5205, Villeurbanne, France\linebreak
\href{mailto:theo.pierron@univ-lyon1.fr}{\tt theo.pierron@univ-lyon1.fr}}

\begin{document}

\maketitle

{\def\thefootnote{}\footnotetext{The present work was initiated when CB was a post-doctoral researcher at CNRS, Universit\'e C\^ote d'Azur, Inria, I3S, Sophia-Antipolis, France.

CB acknowledges the support of the French Agence Nationale de la Recherche (ANR) under contract Digraphs ANR-19-CE48-0013-01.

TP is supported by ANR, under contract GrR ANR-18-CE40-0032.}}

\begin{abstract}
	\noindent
	Given a graph~$G$ and a natural number~$k$, the $k$-recolouring graph~$\mathcal{C}_k(G)$ is the graph whose vertices are the \mbox{$k$-colourings} of~$G$ and whose edges link pairs of colourings which differ at exactly one vertex of~$G$.
	Recently, Hogan \textit{et al.}\ proved that~$G$ can be determined from~$\mathcal{C}_k(G)$ provided~$k$ is large enough (quadratic in the number of vertices of~$G$).
	We improve this bound by showing that $k=\chi(G)+1$ colours suffice, and provide examples of families of graphs for which $k=\chi(G)$ colours do not suffice.

	We then extend this result to $k$-Kempe-recolouring graphs, whose vertices are again the $k$-colourings of a graph~$G$ and whose edges link pairs of colourings which differ by swapping the two colours in a connected component of the subgraph induced by selecting those two colours.
	We show that $k=\chi(G)+2$ colours suffice to determine~$G$ in this case.

	Finally, we investigate the case of independent set reconfiguration, proving that in only a few trivial cases is one guaranteed to be able to determine a graph~$G$.
\end{abstract}

\section{Introduction and main results}

\subsection{Single-vertex recolouring}

Let~$G$ be a graph and~$k$ a positive integer.
A (proper) \emph{$k$-colouring} of~$G$ is an assignment of colours from $[k]:=\{1,2,\ldots,k\}$ to the vertices of~$G$ such that no adjacent vertices receive the same colour.
The \emph{chromatic number} of~$G$, denoted by~$\chi(G)$, is the smallest~$k$ for which~$G$ admits a $k$-colouring.

The \emph{$k$-recolouring graph~$\Ck$} is the graph whose vertices are all the $k$-colourings of~$G$ and whose edges link pairs of $k$-colourings which differ at exactly one vertex.
In this paper we will make no distinction between a vertex of~$\Ck$ and the colouring of~$G$ it represents.

Recolouring graphs have been the object of a growing amount of research.
One of the first aspects that was studied is their use as a tool for generating $k$-colourings (almost) uniformly at random.
This is done by starting from an initial colouring, and then repeatedly uniformly at random selecting a vertex and trying to recolour it with a uniformly random colour from $[k]$.
This process corresponds to following a random walk in~$\Ck$.
For this process to yield an (almost) uniform distribution within a polynomial number of steps, it is necessary (but not sufficient) for~$\Ck$ to be connected and to have polynomial diameter.
Therefore, how these two parameters evolve with respect to the number~$k$ of allowed colours has been the starting point of quite some research.

Since the question of whether a reconfiguration graph is connected or not is essential for many applications, much research has been done also on the computational complexity of this question for different reconfiguration graphs.
Many decision problems of this type are now known to be algorithmically hard (often \PSPACE-complete).
We refer the interested reader to the surveys~\cite{van2013complexity,nishimura2018introduction}.

Other \emph{reconfiguration graphs} have been studied as well; some of them a long time ago.
Reidemeister moves of knot diagrams can be seen as providing edges between different representations of knots.
Furthermore, solving a Rubik's cube consists of finding a path between two configurations in some reconfiguration graph.

\medskip
By definition, giving~$G$ and~$k$ completely determines~$\Ck$.
In early 2024, Asgarli \textit{et al.}~\cite{asgarli2025} asked if the reverse can also be true.
More explicitly, they conjectured that every graph~$G$ is uniquely determined by the collection of recolouring graphs~$\Ck$, $k\in\mathbb{Z}^+$.
Moreover, they conjectured that every graph~$G$ is in fact uniquely determined by some finite subcollection of these recolouring graphs.
Shortly afterwards, these conjectures were answered in a strong sense by Hogan \textit{et al.}~\cite{hogan2024}, who proved that a single recolouring graph~$\Ck$ suffices to determine~$G$, provided~$k$ is large enough.

\begin{theorem}[Hogan \textit{et al.}\ \cite{hogan2024}]
\label{thm:hogan}\mbox{}\\*
	Let~$G$ be an $n$-vertex graph.
	If $k>5n^2$, then~$G$ can be determined from the recolouring graph~$\Ck$, assuming we know~$k$.
\end{theorem}

\noindent
Our first contribution is to improve Theorem~\ref{thm:hogan} by first showing that a much smaller~$k$ suffices, and by removing the assumption that the value of~$k$ is known in advance.

\begin{theorem}
\label{thm:chi+1-general}\mbox{}\\*
	Let~$G$ be a graph.
	If $k>\chi(G)$, then~$G$ can be determined from~$\Ck$, even without knowing the exact value of~$k$.
\end{theorem}

\noindent
Similarly to the result in~\cite{hogan2024}, the proof of Theorem~\ref{thm:chi+1-general} provides a recovery algorithm which runs in polynomial time with respect to the size of~$\Ck$ (which in general is exponential in the size of~$G$).
However, we can prove that by using more colours, we can improve this recovery algorithm to run in polynomial time with respect to $|V(G)|$.

\begin{theorem}
\label{thm:algo}\mbox{}\\*
	Let~$G$ be an $n$-vertex graph of maximum degree $\Delta\ge1$.
	If $k>\min\{n,2\Delta\}$, then~$G$ can be determined from~$\Ck$ by considering at most $(kn)^2$ different colourings in~$\Ck$, even without knowing the exact value of~$k$.
\end{theorem}

\noindent
In addition, if $\Delta=0$, \textit{i.e.}\ if~$G$ has no edges, then our proof gives that~$G$ can be determined easily from~$\Ck$ if $k>1$.
The proofs of Theorems~\ref{thm:chi+1-general} and~\ref{thm:algo} for the most part utilise observations that were already present in the proof of Theorem~\ref{thm:hogan} in~\cite{hogan2024}.
However, the authors in~\cite{hogan2024} use a counting argument, whereas we utilise further structural observations.

The bound on the number of colours in Theorem~\ref{thm:chi+1-general} is tight; a graph~$G$ is not necessarily uniquely determined by its $k$-recolouring graph $\mathcal{C}_k(G)$ if $k=\chi(G)$.
For example, the \mbox{$2$-colouring} graph of any connected bipartite graph is the graph with two vertices and no edges.
And for any $\chi\ge2$, the $\chi$-recolouring graph of any $(\chi-1)$-tree (any graph formed by starting with the complete graph on~$\chi$ vertices, adding a vertex~$v$ such that the neighbours of~$v$ form a clique on $\chi-1$ vertices, and repeating this process finitely many times) is the graph with $\chi!$ vertices and no edges.

These examples exploit the fact that some vertices of the graph~$G$ are \emph{frozen}, \textit{i.e.}\ they cannot be recoloured under any $\chi(G)$-colouring of~$G$.
If a graph~$G$ contains any frozen vertices, then its recolouring graph is non-unique.
To see this, add a new vertex to~$G$ with the same set of neighbours as some frozen vertex~$v$.
This new graph will have the same $\chi(G)$-recolouring graph as~$G$.

However, we will show that the existence of frozen vertices is not necessary for a graph~$G$ to have a non-unique $\chi(G)$-recolouring graph.

\begin{theorem}
\label{thm:counterexample}\mbox{}\\*
	For every $\chi\ge 6$ there exists an arbitrarily large family~$\mathcal{F}$ of non-isomorphic graphs, all with chromatic number~$\chi$ and without frozen vertices, such that $\mathcal{C}_\chi(G)\cong\mathcal{C}_\chi(G')$ for all $G,G'\in \mathcal{F}$.
\end{theorem}

\noindent
Note that all of these results can be easily translated to the case of edge-colourings, since colouring the edges of a graph~$G$ amounts to vertex-colouring the line graph of~$G$.
In particular, one can determine the line graph of~$G$ from its $k$-edge-recolouring graph when~$k$ is larger than the chromatic index of~$G$.
Using a classical theorem of Whitney~\cite{whitney1992}, this allows one to determine~$G$ up to its connected components that are isomorphic to a triangle, a $3$-edge star, or a single vertex.

\paragraph*{Note}\mbox{}\\*
Shortly after we published this paper on arXiv, Asgarli \textit{et al.}~\cite{asgarli2025arxiv} uploaded a paper containing results that are similar to our results in this subsection.
In particular, they prove the non-trivial extension of Theorem~\ref{thm:chi+1-general} that if~$G$ is a graph and $k>\chi(G)$, then for all graphs~$G'$ we have $\mathcal{C}_{\chi(G')}(G')\not\cong\mathcal{C}_k(G)$.

\subsection{Kempe recolouring}

In the previous part we considered reconfiguration graphs obtained by changing the colour of only one vertex at a time.
There is another operation that is natural to consider in problems involving recolouring: performing a Kempe swap.
Given a $k$-colouring of a graph, a \emph{Kempe swap} involves swapping the colours in a maximal connected subgraph induced by two given colours.
Note that recolouring a single vertex is a special case of a Kempe swap, called a \emph{trivial Kempe swap}.

The Kempe swap operation leads to the definition of the \emph{$k$-Kempe-recolouring graph $\KCk$}, whose vertices are all the $k$-colourings of~$G$ and whose edges link $k$-colourings that differ by a single Kempe swap.
Observe that~$\Ck$ and~$\KCk$ have the same vertex set, while the trivial Kempe swaps show that $E(\Ck)\subseteq E(\KCk)$.
As before, we make no distinction between a vertex of~$\KCk$ and the $k$-colouring of~$G$ it represents.

We extend the previous results to the case of Kempe recolouring.
We show that when~$k$ is large enough, we can recognise which edges of~$\KCk$ correspond to trivial Kempe swaps.
In other words, we can recognise which edges of~$\KCk$ induce a copy of~$\Ck$, and thus apply the machinery from the single-vertex recolouring case.

\begin{theorem}
\label{thm:kempe}\mbox{}\\*
	Let~$G$ be a graph.
	If $k>\chi(G)+1$, then~$G$ can be determined from~$\KCk$, even without knowing the exact value of~$k$.
\end{theorem}

\noindent
The same examples used for Theorem~\ref{thm:counterexample} show that the Kempe-recolouring graph is also not unique if $k=\chi(G)$.
This leaves open the case where $k=\chi(G)+1$.

\begin{question}\mbox{}\\*
	\textup{Can every graph~$G$ be determined from $\mathcal{K}_{\chi(G)+1}(G)$?}
\end{question}

\noindent
Similarly to the single-vertex recolouring case, our proof of Theorem~\ref{thm:kempe} provides a recovery algorithm which runs in polynomial time with respect to the size of $\KCk$, which can be improved to run in polynomial time with respect to $|V(G)|$ if the number of colours used is large enough.

\begin{theorem}
\label{thm:algo-kempe}\mbox{}\\*
	Let~$G$ be an $n$-vertex graph of maximum degree $\Delta\ge1$.
	If $k>\min\{n,2\Delta\}+1$, then~$G$ can be determined from~$\KCk$ by considering at most $(kn)^2$ different colourings in~$\KCk$, even without knowing the exact value of~$k$.
\end{theorem}

\subsection{Independent set reconfiguration}

In this last part we investigate the realm of independent set reconfiguration.
An \emph{independent set} in a graph is a set of pairwise non-adjacent vertices.
We interpret each independent set as a set of tokens placed on its vertices, and allow tokens to move according to one of three possible rules (so long as the independent set property is not broken):
\begin{itemize}
	\setlength\itemsep{0mm}
	\item \emph{Token Jumping} (TJ for short): a token can be moved to any other vertex.
	\item \emph{Token Sliding} (TS for short): a token can only be moved to an adjacent vertex.
	\item \emph{Token Addition and Removal} (TAR for short): a single token can be added or removed.
\end{itemize}

\noindent
Given a graph~$G$, we denote by $\IkJ(G)$ the reconfiguration graph of all independent sets of size~$k$ under the TJ rule.
In other words, each vertex of $\IkJ(G)$ corresponds to an independent set of size~$k$ in~$G$, and each edge corresponds to a valid token jump.
We define~$\IkS(G)$ similarly.

Note that TJ and TS preserve the number of tokens, but that this is not the case for TAR.
Therefore, the definition of a TAR-reconfiguration graph is slightly different: $\IkAR(G)$ has a vertex for each independent set of~$G$ of size \emph{at least}~$k$, and each edge corresponds to a single token addition or removal.

It turns out that the situation for independent set reconfiguration is quite different from what we saw previously for recolouring.
Aside from some trivial cases, one is not guaranteed to be able to determine~$G$ from any of the token-reconfiguration graphs defined above.
More precisely, we prove the following.

\begin{theorem}
\label{thm:IS}\mbox{}\\*
	Let~$G$ be a graph.
	If $\,\mathrm{R}=\mathrm{S}$ and $k=1$, or $\,\mathrm{R}=\mathrm{AR}$ and $k\in\{0,1\}$, one can determine~$G$ from $\mathcal{I}^k_\mathrm{R}(G)$.
	In all other cases, there exist non-isomorphic graphs~$G$ with isomorphic independent set reconfiguration graphs.
\end{theorem}

\paragraph*{Organisation of the paper}\mbox{}\\*
Each of the following three sections is devoted to one of the three topics discussed above.
We give proofs for the results mentioned and discuss examples which show the sharpness of these results.

\section{Single-vertex recolouring --- proofs and examples}
\label{section:vertex-recoloration}

We start with the proofs of Theorems~\ref{thm:chi+1-general} and~\ref{thm:algo}.
Let~$G$ be an $n$-vertex graph and let $k>\chi(G)$.
Our approach to determine~$G$ from its $k$-recolouring graph~$\Ck$ is as follows.
Firstly, for every colouring $c\in V(\Ck)$, we use the local structure of~$\Ck$ near~$c$ to construct a candidate graph~$G_c$ for~$G$.
More precisely, each~$G_c$ has a vertex~$v_C$ for each maximal clique~$C$ in the neighbourhood $N_{\Ck}(c)$ of~$c$, and an edge between two vertices~$v_C$ and~$v_{C'}$ if and only if there exist $u\in C$ and $u'\in C'$ whose unique common neighbour in~$\Ck$ is~$c$.
We then show that each of these candidate graphs is a subgraph of~$G$, and that at least one of these candidates is actually~$G$ itself.
This allows us to recognise~$G$ by selecting any candidate graph with the maximum number of vertices and edges.
Moreover, if we further suppose that $k>\min\{n,2\Delta\}$, we show that \emph{all} of these candidate graphs are isomorphic to~$G$, which allows us to determine~$G$ by selecting any colouring $c\in V(\Ck)$ and looking only at the vertices of~$\Ck$ at distance~$1$ and~$2$ from~$c$.

First we investigate the structure of neighbourhoods in~$\Ck$.
We begin with an observation already stated in~\cite{hogan2024}: each clique~$C$ of~$\Ck$ of size at least~$2$ is formed by colourings of~$G$ which differ only at a single vertex of~$G$.
We denote this vertex by~$v_C$.
Using this property, we observe that the neighbourhood of each colouring~$c$ in~$\Ck$ induces a union of~$n$ (possibly empty) cliques, each corresponding to a vertex of~$G$.
For each~$c\in V(\Ck)$ and~$u\in V(G)$, we denote by~$C_{c,u}$ the clique in~$\Ck$ containing all colourings which differ from~$c$ only at~$u$ (where~$C_{c,u}$ is empty if no such colourings exist).
For each colouring $c\in V(\Ck)$, the number of vertices in the candidate graph~$G_c$ will be the number of connected components in its neighbourhood, which is at most~$n$.

To determine the edges of~$G$, we will consider the neighbourhoods at distance~$2$ in~$\Ck$.

\begin{lemma}
\label{lem:chi+1-general}\mbox{}\\*
	Let~$c$ be a $k$-colouring of~$G$.
	If $uv\notin E(G)$, then every $c_u\in V(C_{c,u})$ and $c_v\in V(C_{c,v})$ have a common neighbour distinct from~$c$ $(\star)$.
\end{lemma}

\begin{proof}
	Given $c_u\in V(C_{c,u})$ and $c_v\in V(C_{c,v})$, define a $k$-colouring~$c'$ of~$G$ as follows: $c'(u)=c_u(u)$, $c'(v)=c_v(v)$, and $c'(w)=c(w)$ for every $w\notin\{u,v\}$.
	Note that~$c'\ne c$, and that~$c'$ differs from~$c_u$ (resp.\ from~$c_v$) only at~$v$ (resp.\ only at~$u$).
	Moreover, since~$u$ and~$v$ are not adjacent, $u$ and all of its neighbours are assigned the same colours under~$c'$ and~$c_v$, and~$v$ and all of its neighbours are assigned the same colours under~$c'$ and~$c_u$.
	Since~$c_u$ and~$c_v$ are proper colourings of~$G$, we conclude that~$c'$ is a proper colouring of~$G$.
\end{proof}

\noindent
Condition $(\star)$ in Lemma~\ref{lem:chi+1-general} gives a partial criterion to determine the edges of~$G$.
We now consider each colouring~$c$ in~$\Ck$ and construct a corresponding candidate graph~$G_c$ for~$G$ as follows.
For each $c\in V(\Ck)$, let~$G_c$ be the graph with a vertex~$v_C$ for each maximal clique~$C$ of~$N_{\Ck}(c)$, and an edge between two vertices~$v_C$ and~$v_{C'}$ if and only if the cliques~$C$ and~$C'$ do not satisfy $(\star)$ (that is, if and only if there exist $a\in C$ and $a'\in C'$ which have no common neighbour distinct from~$c$).
By Lemma~\ref{lem:chi+1-general}, each non-edge of~$G$ is also a non-edge of~$G_c$, which gives the following corollary.

\begin{corollary}
\label{cor:chi+1-general}\mbox{}\\*
	For each $k$-colouring~$c$ of~$G$, $G_c$ is a subgraph of~$G$.
\end{corollary}

\noindent
What is left to show is that~$G_c$ is isomorphic to~$G$ for at least one choice of~$c$.
The argument for this will come from the following observation.

\begin{lemma}
\label{lem:exact-chi+1}\mbox{}\\*
	Let $uv\in E(G)$ and let~$c$ be a $k$-colouring of~$G$ such that there is a colour $i\in[k]$ which does not appear in $\{u,v\}\cup N(u)\cup N(v)$.
	Then $(\star)$ does not hold.
\end{lemma}

\begin{proof}
	Define~$c_u$ (resp.~$c_v$) as the colouring obtained from~$c$ by recolouring~$u$ (resp.~$v$) to colour~$i$.
	Note that~$c_u$ and~$c_v$ differ on exactly~$u$ and~$v$, hence any colouring $c'\ne c$ adjacent to both~$c_u$ and~$c_v$ must assign colour~$i$ to both~$u$ and~$v$.
	But no such colouring can exist since~$u$ and~$v$ are adjacent.
\end{proof}

\noindent
If we have $k>\chi(G)$, then for any $k$-colouring~$c'$ of~$G$ using only $\chi(G)$ colours, there exists a colour $i\in[k]$ completely unused by~$c'$.
Combining Corollary~\ref{cor:chi+1-general} and Lemma~\ref{lem:exact-chi+1}, we deduce that~$G_{c'}$ is isomorphic to~$G$ for any such~$c'$.
Moreover, since Corollary~\ref{cor:chi+1-general} tells us that~$G_c$ is a subgraph of~$G$ for every $c\in V(\Ck)$, we can recognise such a~$c'$ by considering the set $\{G_c \mid c\in V(\Ck)\}$ of all candidate graphs and selecting any one with the maximum number of vertices and edges.
This completes the proof of Theorem~\ref{thm:chi+1-general}.

Moreover, if we further suppose $k>\min\{n,2\Delta\}$, then the condition in Lemma~\ref{lem:exact-chi+1} is satisfied by every $k$-colouring~$c$ of~$G$.
Therefore, $G_c$ is isomorphic to~$G$ for every~$c$ in~$\Ck$ in this case.
Given a colouring~$c$ in~$\Ck$, to determine~$G_c$ we only need to consider the colourings at distance~$1$ and~$2$ from~$c$ in~$\Ck$.
Observing that any colouring has at most $(k-1)n$ neighbours in~$\Ck$, we can therefore determine~$G$ in this case by considering at most $(kn)^2$ colourings in~$\Ck$, completing the proof of Theorem~\ref{thm:algo}.

\subsection{The case \texorpdfstring{$k=\chi(G)$}{k=chi}}
\label{sub:2.1}

The situation becomes very different when $k=\chi(G)$.
First of all, in some cases (for example if~$G$ is a clique) the graph~$G$ may contain vertices whose closed neighbourhoods contain all of the available~$k$ colours under every proper $k$-colouring of~$G$, and thus can never be recoloured.
We call these vertices \emph{frozen}.
If a vertex can change colour in at least one $\chi(G)$-colouring of~$G$ we call it \emph{recolourable}.
In the introduction, we showed that if a graph contains frozen vertices, then its $\chi(G)$-recolouring graph is not unique.
In the next theorem we will show that this is not a necessary condition.

\begin{manualtheorem}{\ref{thm:counterexample}}\mbox{}\\*
	For every $\chi\ge 6$ there exists an arbitrarily large family~$\mathcal{F}$ of non-isomorphic graphs, all with chromatic number~$\chi$ and without frozen vertices, such that $\mathcal{C}_\chi(G)\cong\mathcal{C}_\chi(G')$ for all $G,G'\in \mathcal{F}$.
\end{manualtheorem}

\noindent
Our construction relies on the following classical construction of Mycielski~\cite{myc1955}.
Let~$G$ be an $n$-vertex graph with vertices $v_1,\ldots,v_n$.
Then the \emph{Mycielskian $M(G)$ of~$G$} is constructed by
adding vertices $u_1,\ldots,u_n$ and~$w$, adding edges between~$w$ and each~$u_i$, and for each $v_iv_j\in E(G)$ adding the edges~$u_iv_j$ and~$v_iu_j$.
We use the notation $M^2(G)=M(M(G))$, and so on.
The essential property of this construction is that $\chi(M^k(G))=\chi(G)+k$ for every~$k$ and graph~$G$~\cite{myc1955}.

We can now give our construction to provide the examples promised in Theorem~\ref{thm:counterexample}.

\begin{construction}
\label{ex:nofrozen}\mbox{}\\*
	\textup{Let $\chi\ge 6$ and $p\in[3,\chi-3]$.
	Let~$H_0$ and~$H_3$ be any two graphs such that $\chi(H_0)<\chi-p$ and $\chi(H_3)<p$.
	Define $H_1=M^{p-3}(C_5)$ and $H_2=M^{\chi-p-3}(C_5)$, where~$C_5$ is the cycle on~$5$ vertices.}

	\textup{Let~$G_0$ be the graph formed by starting with the disjoint union of $H_0,H_1,H_2,H_3$ and adding all edges of the form $h_ih_{i+1}$, where $h_i\in H_i$, $h_{i+1}\in H_{i+1}$, and $i\in\{0,1,2\}$.
	Finally, let~$\mathcal{F}$ be the set of all non-isomorphic graphs that can be obtained from~$G_0$ by adding any collection of edges between~$H_0$ and~$H_3$.}

	\textup{Observe that since~$H_0$ and~$H_3$ can be chosen to be arbitrarily large,~$\mathcal{F}$ can be arbitrarily large.}
\end{construction}

\noindent
The graph~$G_0$ defined in Construction~\ref{ex:nofrozen} is sketched in Figure~\ref{fig:nofrozen}.
Grey zones indicate that all edges between the two graphs are present in~$G_0$.

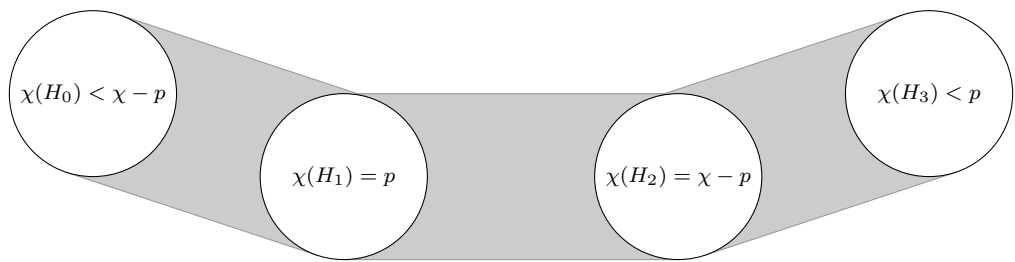
\begin{figure}[ht]
	\centering
	\begin{tikzpicture}[scale=1.1]
		\usetikzlibrary{calc}

		\node (A) at (-2, 0) {};
		\node (B) at (2, 0) {};
		\node (C) at (-5,1) {};
		\node (D) at (5,1) {};

		\pgfmathanglebetweenpoints{\pgfpointanchor{A}{center}}{\pgfpointanchor{C}{center}}\let\angleAC\pgfmathresult
		\pgfmathanglebetweenpoints{\pgfpointanchor{B}{center}}{\pgfpointanchor{D}{center}}\let\angleBD\pgfmathresult
		\coordinate (offset0) at (0,1);
		\coordinate (offset1) at ({cos(\angleAC+90)},{sin(\angleAC+90)});
		\coordinate (offset2) at ({cos(\angleBD+90)},{sin(\angleBD+90)});

		\draw[white!60!black, fill=white!80!black] ($ (A) + (offset0) $) -- ($ (B) + (offset0) $) -- ($ (B) - (offset0) $) -- ($ (A) - (offset0) $) -- cycle;

		\draw[white!60!black, fill=white!80!black] ($ (A) + (offset1) $) -- ($ (C) + (offset1) $) -- ($ (C) - (offset1) $) -- ($ (A) - (offset1) $) -- cycle;

		\draw[white!60!black, fill=white!80!black] ($ (B) + (offset2) $) -- ($ (D) + (offset2) $) -- ($ (D) - (offset2) $) -- ($ (B) - (offset2) $) -- cycle;

		\draw[fill=white, draw=black] (A) circle [radius=1];
		\draw[fill=white, draw=black] (B) circle [radius=1];
		\draw[fill=white, draw=black] (C) circle [radius=1];
		\draw[fill=white, draw=black] (D) circle [radius=1];

		\node at (A) {\scriptsize$\chi(H_1)=p$};
		\node at (B) {\scriptsize$\chi(H_2)=\chi-p$};
		\node at (C) {\scriptsize$\chi(H_0)<\chi-p$};
		\node at (D) {\scriptsize$\chi(H_3)<p$};
	\end{tikzpicture}

	\caption{The graph~$G_0$ with chromatic number~$\chi$ from Construction~\ref{ex:nofrozen}.
	Grey zones indicate complete joins.}
	\label{fig:nofrozen}
\end{figure}

The main idea of Construction~\ref{ex:nofrozen} is as follows.
The graph~$G_0$ has chromatic number~$\chi$ and no frozen vertices.
On the other hand, under each $\chi$-colouring of~$G_0$, the colours assigned to the vertices in~$H_0$ are disjoint from the colours assigned to the vertices in~$H_3$.
This latter property means that adding edges between~$H_0$ and~$H_3$ does not affect the $\chi$-recolouring graph of~$G_0$.

To prove that the set~$\mathcal{F}$ satisfies the requirement for Theorem~\ref{thm:counterexample} we need one further observation.

\begin{lemma}
\label{lem:mycielskian2}\mbox{}\\*
	Let~$G$ be a graph with no frozen vertices.
	Then the Mycielskian $M(G)$ of~$G$ has no frozen vertices.
\end{lemma}

\begin{proof}
	As in the definition of the Mycielskian of a graph, we denote the vertices of~$G$ by $v_1,\ldots,v_n$ and the vertices of $M(G)$ by $v_1,\ldots,v_n,u_1,\ldots,u_n,w$.
	Let $\chi=\chi(G)$ and fix $i\in [n]$.
	Since~$G$ has no frozen vertices, there is a $\chi$-colouring~$c$ of~$G$ for which~$v_i$ is recolourable.
	Let~$c^M$ be the colouring of $M(G)$ where $c^M(v_j)=c^M(u_j)=c(v_j)$ for each $j\in[n]$ and $c^M(w)=\chi+1$.
	Since $\chi(M(G))=\chi(G)+1$, this colouring of $M(G)$ uses the minimum number of colours.
	Notice that~$v_i$ and~$u_i$ are recolourable in~$c^M$.
	Since $i\in[n]$ was arbitrary, none of $v_1,\ldots,v_n,u_1,\ldots,u_n$ are frozen.

	Next let~$c$ be any colouring of~$G$, and let~$c^M$ be the colouring of $M(G)$ given by $c^M(v_j)=c(v_j)$ for each $j\in[n]$, $c^M(u_j)=\chi+1$ for each $j\in[n]$, and $c^M(w)=1$.
	Then~$w$ is recolourable in~$c^M$, hence~$w$ is not frozen.
\end{proof}

\begin{proof}[Proof of Theorem \ref{thm:counterexample}]
	Let $\chi\ge6$, $p$, $H_0,H_1,H_2,H_3$, $G_0$, and~$\mathcal{F}$ be as in Construction~\ref{ex:nofrozen}.
	As observed already, $\mathcal{F}$ can be made arbitrarily large since~$H_0$ and~$H_3$ can be made arbitrarily large.

    Since $\chi(C_5)=3$, for $H_1,H_2$ we have $\chi(H_1)=\chi(M^{p-3}(C_5))=p$ and $\chi(H_2)=\chi(M^{\chi-p-3}(C_5))=\chi-p$.
    Hence every $G\in\mathcal{F}$ has chromatic number~$\chi$ by construction.

	We next prove that for each $G\in\mathcal{F}$, all vertices of~$G$ are recolourable.
	Observe that the cycle~$C_5$ of length~$5$ has no frozen vertices.
	Hence $H_1=M^{p-3}(C_5)$ and $H_2=M^{\chi-p-3}(C_5)$ have no frozen vertices by Lemma~\ref{lem:mycielskian2}.

	Fix $h_1\in H_1$ and let $c_1:V(H_1)\to [p]$ be a \mbox{$p$-colouring} such that~$h_1$ is recolourable to some $i\in [p]$.
	Fix $h_2\in H_2$ and let $c_2:V(H_2)\to [p+1,\chi]$ be a $(\chi-p)$-colouring such that~$h_2$ is recolourable to some $j\in [p+1,\chi]$.
	Consider any $(\chi-p-1)$-colouring $c_0:V(H_0)\to [p+1,\chi-1]$ of~$H_0$ and any $(p-1)$-colouring $c_3:V(H_3)\to [p-1]$ of~$H_3$.
	Let~$G$ be a graph in~$\mathcal{F}$ (hence $V(G)=V(H_0)\cup\cdots\cup V(H_3)$) and form the $\chi$-colouring~$c$ of~$G$ by copying $c_0,\ldots,c_3$ for the different parts of~$G$.

	Note that no matter what edges there are between~$H_0$ and~$H_3$ in~$G$, every vertex in~$H_0$ is recolourable to~$\chi$ and every vertex in~$H_3$ is recolourable to~$p$.
	Moreover, $h_1$ and~$h_2$ are recolourable to~$i$ and~$j$, respectively.
	Since~$h_1$ and~$h_2$ were arbitrary elements of~$H_1$ and~$H_2$ respectively, we have shown that all vertices of any $G\in\mathcal{F}$ are recolourable.

	It remains to prove that every graph in~$\mathcal{F}$ has the same $\chi$-recolouring graph.
	We use the following claim.

	\begin{claim}
		Let~$G$ be a graph with chromatic number~$\chi$ and suppose $uv\notin E(G)$ satisfies $c(u)\ne c(v)$ for all $\chi$-colourings~$c$ of~$G$.
		Let~$G'$ be the graph formed from~$G$ by adding the edge $uv$.
		Then $\mathcal{C}_\chi(G)\cong\mathcal{C}_\chi(G')$.
	\end{claim}

	\begin{proof}[Proof of claim]
		Clearly $\chi(G')=\chi(G)$ and $\mathcal{C}_\chi(G')$ is a subgraph of $\mathcal{C}_\chi(G)$.
		It remains to show that $\mathcal{C}_\chi(G)$ is also a subgraph of $\mathcal{C}_\chi(G')$.
		Let $c\in V(\mathcal{C}_\chi(G))$.
		Then $c(u)\ne c(v)$, so~$c$ is a proper colouring of~$G'$, hence $c\in V(\mathcal{C}_\chi(G'))$, which gives $V(\mathcal{C}_\chi(G))\subseteq V(\mathcal{C}_\chi(G'))$.
		Now let $c_1c_2\in E(\mathcal{C}_\chi(G))$.
		Then~$c_1$ and~$c_2$ are $\chi$-colourings of~$G$ which differ at some vertex~$w$ of~$G$ and nowhere else, and so (since $V(\mathcal{C}_\chi(G))\subseteq  V(\mathcal{C}_\chi(G'))$), $c_1$ and~$c_2$ are $\chi$-colourings of~$G'$ which differ at some vertex~$w$ of~$G'$ and nowhere else.
		This means $c_1c_2\in E(\mathcal{C}_\chi(G'))$, hence  $E(\mathcal{C}_\chi(G))\subseteq E(\mathcal{C}_\chi(G'))$ as required.
	\end{proof}

	\noindent
	Let~$c$ be a $\chi$-colouring of~$G_0$.
	For $i\in\{1,2\}$, let~$A_i$ be the set of colours assigned to the vertices of~$H_i$ by~$c$, \textit{i.e.}\ $A_i=\{j\in [\chi]\mid \exists\, w\in V(H_i),\ c(w)=j\}$.
	Observe that $A_1\cap A_2=\varnothing$, since every vertex of~$H_1$ is adjacent to every vertex of~$H_2$.
	We also have $|A_2|\ge \chi-p$, since~$H_1$ and~$H_2$ have chromatic numbers~$p$ and $\chi-p$ respectively.
	This gives $A_1\cup A_2=[\chi]$, $|A_1|=p$, and $|A_2|=\chi-p$.
	Let $u\in H_0$ and $v\in H_3$.
	Then $c(u)\notin A_1$, since~$u$ is adjacent to every vertex in~$H_1$; hence $c(u)\in A_2$.
	Similarly  we have $c(v)\in A_1$.
	This gives $c(u)\ne c(v)$.

	Hence, under every $\chi$-colouring of~$G_0$, every vertex of~$H_0$ is assigned a different colour to every vertex of~$H_3$.
	By the claim above, this means that any graph~$G$ formed from~$G_0$ by adding edges between vertices of~$H_0$ and~$H_3$ will have the same $\chi$-recolouring graph as~$G_0$, completing the proof of the theorem.
\end{proof}

\section{Kempe recolouring --- proofs and examples}
\label{sec:kempe}

For a $k$-colouring~$c$ of~$G$ and $u\in V(G)$, we denote by~$C_{c,u}$ the set of colourings in~$\KCk$ obtained from~$c$ by a trivial Kempe swap on~$u$.
The trivial Kempe swaps correspond to the single-vertex recolouring operation discussed in Section~\ref{section:vertex-recoloration}, therefore~$C_{c,u}$ is again a clique for every $u\in V(G)$.
We denote by~$D_c$ the set of other neighbours of~$c$ in~$\KCk$, \textit{i.e.}\ those vertices corresponding to colourings of~$G$ obtained from~$c$ by a non-trivial Kempe swap.

Similarly to the case of single-vertex recolouring, we first study the graph induced by the neighbourhood of each colouring~$c$ in~$\KCk$.
We show that, as before, each connected component in each such graph is a clique, and that any such clique which was not already in the neighbourhood of~$c$ in~$\Ck$ is an isolated vertex.

\begin{lemma}
\label{lm:cliques-Kempe}\mbox{}\\*
	For any $k$-colouring~$c$ of~$G$ and distinct $u,v\in V(G)$, there is no edge in~$\KCk$ between~$C_{c,u}$ and~$C_{c,v}$.
\end{lemma}

\begin{proof}
	Let $c_u\in C_{c,u}$ and $c_v\in C_{c,v}$.
	Suppose $c_uc_v\in E(\KCk)$.
	Then we can move from~$c_u$ to~$c_v$ via a Kempe swap.
	Since $c_u$ and $c_v$ differ on $u$ and $v$ and agree on all other vertices of $G$, this Kempe swap affects~$u$ and~$v$ only.
	This means that $\{u,v\}$ must induce a connected subgraph of~$G$, and hence $uv\in E(G)$.
	On the other hand, we have $c_v(u)=c(u)$ and $c_u(v)=c(v)$.
	And after the Kempe swap from~$c_u$ to~$c_v$,~$u$ has colour~$c_u(v)$, \textit{i.e.}\ we have $c_v(u)=c_u(v)$.
	But this means that $c(u)=c(v)$, hence $uv\notin E(G)$, a contradiction.
\end{proof}

\begin{lemma}
\label{lm:independent}\mbox{}\\*
	For any $k$-colouring~$c$ of~$G$ and $u\in V(G)$, $D_c$ is an independent set of~$\KCk$ with no neighbour in~$C_{c,u}$.
\end{lemma}

\begin{proof}
	Let $c_1\in D_c$ and $c_2\in D_c\cup C_{c,u}$, where $c_1\neq c_2$.
	Let $S_1$, resp.\ $S_2$, denote the set of colours involved in the Kempe swap performed between $c$ and $c_1$, resp.\ between~$c$ and~$c_2$ (in the case where $c_2\in C_{c,u}$, this means $S_2=\{c(u),c_2(u)\}$).
	Then $|S_1|=|S_2|=2$.
	If $|S_1\cup S_2|\ge3$, then at least three colours appear on the set of vertices where~$c_1$ and~$c_2$ differ, hence we cannot move from~$c_1$ to~$c_2$ via a single Kempe swap.
	So we have $S_1=S_2$.
	Suppose $c_2\in D_c$.
	Then, since $c_1\ne c_2$, the subgraph of~$G$ induced by the set of vertices where~$c_1$ and~$c_2$ differ consists of exactly two connected components.
	Hence we cannot move from~$c_1$ to~$c_2$ via a Kempe swap in this case.
	So we have $c_2\in C_{c,u}$.
	Set $i\coloneqq c_2(u)$, and let $S_1=S_2=\{i,j\}$.
	Let~$A$ denote the set of vertices of~$G$ that change colour in the Kempe swap from~$c$ to~$c_1$.
	Since $|A|\ge2$, every $v\in A$ either satisfies $c(v)=i$, or has a neighbour $w\in A$ such that $c(w)=i$.
	Therefore we must have $u\notin A$, since moving from~$c$ to~$c_2$ changes the colour of~$u$ to~$i$.
	Since $S_1=S_2$, this means that there is no edge between~$u$ and~$A$.
	Therefore, the subgraph of~$G$ induced by the set of vertices where~$c_1$ and~$c_2$ differ consists of exactly two connected components, hence we cannot move from~$c_1$ to~$c_2$ via a Kempe swap in this case either.
\end{proof}

\begin{lemma}
\label{lm:are-trivial}\mbox{}\\*
	Let~$c$ be a $k$-colouring of~$G$, let $u, v\in V(G)$ be such that $|C_{c,u}|,|C_{c,v}|\ge2$, and let $c_u\in C_{c,u}$ and $c_v\in C_{c,v}$.
	If $uv \notin E(G)$, then $|C_{c_u,v}|\ge2$ and $|C_{c_v,u}|\ge2$.
\end{lemma}

\begin{proof}
	Since $|C_{c,v}|\ge2$, there are at least two colours available to recolour~$v$ in~$c$.
	However, $c$ and~$c_u$ only differ on~$u$, which we assume is not adjacent to~$v$.
	Therefore we can also recolour~$v$ in~$c_u$ by these two colours, hence $|C_{c_u,v}|\ge2$.
	A similar argument shows that $|C_{c_v,u}|\ge2$.
\end{proof}

\begin{lemma}
\label{lm:chi+1-general-Kempe}\mbox{}\\*
	Let~$c$ be a $k$-colouring of~$G$.
	If $uv \notin E(G)$, then for every $c_u\in C_{c,u}$ and $c_v\in C_{c,v}$, we have $C_{c_u,v} \cap C_{c_v,u}\ne \varnothing$.
\end{lemma}

\begin{proof}
	The transitions between~$c$ and~$c_u$ and between~$c$ and~$c_v$ are trivial Kempe swaps, thus by Lemma~\ref{lem:chi+1-general}, $C_{c_u,v} \cap C_{c_v,u}$ is non-empty.
\end{proof}

\noindent
Given a $k$-colouring~$c$ of~$G$ and $u,v\in V(G)$, $u\ne v$, such that $|C_{c,u}|,|C_{c,v}|\ge2$, we say that $c_u\in C_{c,u}$ and $c_v\in C_{c,v}$ satisfy condition $(\star\star)$ if~$c_u$ and~$c_v$ have a common neighbour~$x$ other than~$c$, and there exist colourings $y_1,y_2$ such that $\{x,c_u,y_1\}$ and $\{x,c_v,y_2\}$ both induce triangles in~$\KCk$.

The following lemma tells us that condition $(\star\star)$ being satisfied is equivalent to the conclusions of Lemmas~\ref{lm:are-trivial} and~\ref{lm:chi+1-general-Kempe} both holding.

\begin{lemma}
\label{lm:kempecriterion}\mbox{}\\*
	Let~$c$ be a $k$-colouring of~$G$ and let $u,v\in V(G)$, $u\ne v$, be such that $|C_{c,u}|,|C_{c,v}|\ge2$.
	Then $c_u\in C_{c,u}$ and $c_v\in C_{c,v}$ satisfy $(\star\star)$ if and only if $|C_{c_u,v}|\ge2$, $|C_{c_v,u}|\ge2$, and $C_{c_u,v} \cap C_{c_v,u}\ne \varnothing$.
\end{lemma}

\begin{proof}
	If $(\star\star)$ holds, then by Lemmas~\ref{lm:cliques-Kempe} and~\ref{lm:independent} we must have $x,y_1\in C_{c_u,v}$ and $x,y_2\in C_{c_v,u}$.
	Conversely, if $|C_{c_u,v}|\ge2$, $|C_{c_v,u}|\ge2$, and $C_{c_u,v} \cap C_{c_v,u}\ne \varnothing$, then $(\star\star)$ holds by choosing any $x\in C_{c_u,v} \cap C_{c_v,u}$, $y_1\in C_{c_u,v}\backslash\{x\}$, and $y_2\in C_{c_v,u}\backslash\{x\}$.
\end{proof}

\noindent
We can now adapt the proof of Theorem~\ref{thm:chi+1-general} to obtain Theorem~\ref{thm:kempe}.
Similarly to before, we will use each vertex~$c$ of~$\KCk$ to construct a candidate graph~$G_c$ for~$G$ by looking at its neighbourhood at distance~$2$, as follows.
Fix $c\in \KCk$.
Create a vertex~$u$ in~$G_c$ for each maximal clique~$C_{c,u}$ of size at least~$2$, so that $V(G_c) \subseteq V(G)$.
(Observe that by Lemma~\ref{lm:independent} we can recognise such~$C_{c,u}$.)
Then put an edge between two vertices~$u$ and~$v$ in~$G_c$ if and only if there exists a pair $c_u\in C_{c,u}$ and $c_v\in C_{c,v}$ such that condition $(\star\star)$ is \textbf{not} satisfied.
(Observe that this is something we can easily check.)
By Lemmas~\ref{lm:are-trivial}, \ref{lm:chi+1-general-Kempe}, and~\ref{lm:kempecriterion}, if $uv\in E(G_c)$, then $uv\in E(G)$.
In other words, each candidate graph~$G_c$ is a subgraph of~$G$.

Unfortunately, ignoring cliques of size~$1$ may remove some~$C_{c,u}$ where~$u$ sees exactly $k-1$ colours in its closed neighbourhood.
However, we can still show the following.

\begin{lemma}
\label{lem:3.5}\mbox{}\\*
	If $k>\chi(G)+1$, then every colouring~$c$ in~$\KCk$ that uses at most $k-2$ colours satisfies $G_c\cong G$.
\end{lemma}

\begin{proof}
	Let $c\in V(\KCk)$ be a colouring that uses at most $k-2$ colours, and let $i,j\in[k]$ be colours not used by~$c$.
	Observe that every vertex~$u$ of~$G$ can be recoloured via a trivial Kempe swap to both colours~$i$ and~$j$.
	Hence $|C_{c,u}|\ge 2$ for every $u\in V(G)$, and so $V(G_c)=V(G)$.
	Since we know that~$G_c$ is a subgraph of~$G$, it remains to show that every edge of~$G$ lies in~$G_c$.
	Suppose $uv\in E(G)$.
	Let~$c_u$ (resp.~$c_v$) be the colouring where~$u$ (resp.~$v$) has been recoloured with colour~$i$.
	If there exists some $c'\in C_{c_u,v} \cap C_{c_v,u}$, then $c'(u)=i=c'(v)$, a contradiction.
	So $C_{c_u,v} \cap C_{c_v,u}$ is empty.
	Therefore, by Lemma~\ref{lm:kempecriterion}, $c_u$ and~$c_v$ do not satisfy $(\star\star)$.
	Hence~$uv\in E(G_c)$.
\end{proof}

\noindent\noindent
In particular, if $k>\chi(G)+1$, then applying Lemma~\ref{lem:3.5} to any $k$-colouring~$c$ which uses only $\chi(G)$ colours yields a candidate graph~$G_c$ isomorphic to~$G$.
As with the case of single-vertex recolouring, since all candidate graphs are subgraphs of~$G$, we can recognise candidate graphs isomorphic to~$G$ by selecting any candidate graph with the maximum number of vertices and edges.
This completes the proof of Theorem~\ref{thm:kempe}.

Proving Theorem~\ref{thm:algo-kempe} requires only a small amount of additional work.

\begin{proof}[Proof of Theorem \ref{thm:algo-kempe}]
	If $k>\min\{n,2\Delta\}+1$, then under any $k$-colouring of~$G$, every vertex of~$G$ can be recoloured with at least two colours.
	Arguing as in the proof of Lemma~\ref{lem:3.5} tells us that~$G_c$ is isomorphic to~$G$ for every $k$-colouring~$c$ of~$G$.
	Given such a $k$-colouring~$c$, every Kempe swap from~$c$ is determined by choosing a vertex~$v$ and a colour $j\ne c(v)$: swap the colours $c(v)$ and~$j$ on the maximal induced subgraph with colours~$c(v)$ and~$j$ containing~$v$.
	This means that there are at most $kn$ Kempe swaps possible from any $k$-colouring of~$G$.
	Hence any $k$-colouring of~$G$ has at most $(kn)^2$ neighbours at distance~$1$ or~$2$ in~$\KCk$, completing the proof of Theorem~\ref{thm:algo-kempe}.
\end{proof}

\noindent
The same examples from Subsection~\ref{sub:2.1} for single-vertex recolouring show that a graph~$G$ cannot in general be determined from its $\chi(G)$-Kempe-recolouring graph.
For example, the $k$-Kempe-recolouring graph of any $(k-1)$-tree is the graph on~$k!$ vertices, each corresponding to some permutation of $\{1,\ldots,k\}$, where two vertices share an edge if and only if their corresponding permutations differ by a transposition of two elements (colours).
Similarly, the construction from Theorem~\ref{thm:counterexample} applies to the case of Kempe-recolouring by observing that, as before, given a graph~$G$, adding an edge between two vertices which can never be assigned the same colour under any $\chi(G)$-colouring cannot change its $\chi(G)$-Kempe-recolouring graph, even now that we allow non-trivial Kempe swaps.

However, the question of whether any graph~$G$ can be determined from its $(\chi(G)+1)$-Kempe-recolouring graph remains open.

\section{Independent set reconfiguration --- proofs and examples}
\label{sec:IS}

In this section we investigate the possibility of determining a graph~$G$ from its independent set reconfiguration graph by proving Theorem~\ref{thm:IS}.
We begin by considering the three cases (TJ, TS, and TAR) for $k\in\{0,1\}$.

The TS and TJ cases for $k\in\{0,1\}$ follow from the following straightforward observations.

\begin{observation}\mbox{}\\*
	Let~$G$ be a graph.
	Then $\mathcal{I}^0_\mathrm{J}(G)$ and $\mathcal{I}^0_\mathrm{S}(G)$ are single vertices, $\mathcal{I}^1_\mathrm{J}(G)$ is a clique, and~$\mathcal{I}^1_\mathrm{S}(G)$ is isomorphic to~$G$.
\end{observation}

\noindent
The TAR case requires a bit more work, hence we dedicate separate parts to the two cases $k=0$ and $k=1$.

\subsection{Token Addition and Removal with \texorpdfstring{$k=0$}{k=0}}

To determine~$G$ from its reconfiguration graph $\mathcal{I}^0_\mathrm{AR}(G)$, we will show that all the information needed is contained in the neighbourhood of any maximum-degree vertex.
First note that the degree of a vertex in $\mathcal{I}^0_\mathrm{AR}(G)$ is at most $n=|V(G)|$, since we can only add or remove one token at a time.
Since the empty set has exactly~$n$ neighbours, we find that the maximum degree of $\mathcal{I}^0_\mathrm{AR}(G)$ is~$n$.

We start by characterising the vertices that have maximum degree in $\mathcal{I}^0_\mathrm{AR}(G)$.
From the observation above, the neighbours of a maximum-degree independent set of~$G$ correspond exactly to the addition or removal of a token on a vertex of~$G$.
Therefore, we can use $V(G)$ to index the neighbourhood of any vertex~$I$ of $\mathcal{I}^0_\mathrm{AR}(G)$ of degree~$n$.
The neighbourhood of such an~$I$ is the family $(I_x)_{x\in V(G)}$, where~$I_x$ denotes the symmetric difference $I\triangle\{x\}$ (the independent set obtained from~$I$ by adding or removing a token on the vertex~$x$ of~$G$).

\begin{lemma}\mbox{}\\*
	Let~$I$ be a vertex of $\mathcal{I}^0_\mathrm{AR}(G)$ of maximum degree and let~$I_u$ and~$I_v$ be two neighbours of~$I$.
	Then $uv\notin E(G)$ if and only if~$I_u$ and~$I_v$ have a common neighbour $I'\ne I$ in $\mathcal{I}^0_\mathrm{AR}(G)$.
\end{lemma}

\begin{proof}
	($\Rightarrow$) Suppose $uv\notin E(G)$ and let $I'=I\triangle\{u,v\}=I_u\triangle\{v\}=I_v\triangle\{u\}$.
	The set~$I'$ is an independent set since $uv\notin E(G)$.
	Moreover, $I'\ne I$, and~$I_u$ and~$I'$ differ only in the vertex~$v$, and so $I_uI'\in E(\mathcal{I}^0_\mathrm{AR}(G))$.
	By symmetry of~$u$ and~$v$, we also have $I_vI'\in E(\mathcal{I}^0_\mathrm{AR}(G))$.

	($\Leftarrow$) Now suppose $uv\in E(G)$.
	Then~$u$ and~$v$ cannot belong simultaneously to an independent set of~$G$, so $I,I_u,I_v$ cannot contain both~$u$ and~$v$.
	Since either~$I$ or~$I_u$ contains~$u$, none of them may contain~$v$, and similarly~$I$ and~$I_v$ do not contain~$u$.
	Therefore $u\in I_u$ and $v\in I_v$.
	But then any neighbour $I'\ne I$ of~$I_u$ and~$I_v$ in $\mathcal{I}^0_\mathrm{AR}(G)$ must contain both~$u$ and~$v$, a contradiction.
\end{proof}

\noindent
The above lemma yields an algorithm to determine~$G$ from its reconfiguration graph $\mathcal{I}^0_\mathrm{AR}(G)$.
This algorithm runs in time polynomial in $|\mathcal{I}^0_\mathrm{AR}(G)|$.
However, the problem becomes more complicated when the empty independent set is not allowed, as we will see in the next part.

\subsection{Token Addition and Removal with \texorpdfstring{$k=1$}{k=1}}

Determining~$G$ from $\mathcal{I}^1_\mathrm{AR}(G)$ relies on determining, for each vertex of $\mathcal{I}^1_\mathrm{AR}(G)$, the size of the corresponding independent set of~$G$.
Indeed, once this is known, the graph~$G$ can be determined by associating a vertex with each independent set of size~$1$, then adding edges when these independent sets do not share a neighbour of size~$2$.

To this end, we introduce \emph{layerings} of graphs.
Given a graph~$G$, a layering of~$G$ is a partition of $V(G)$ into parts $V_1,\ldots,V_\ell$ (each part~$V_i$ being a \emph{layer}) such that:
\begin{enumerate}
	\setlength\itemsep{0mm}
	\item Each layer is an independent set of~$G$.\label{cond:indep}
	\item Each layer only has neighbours in the layers directly above and below it.\label{cond:neighbours}
	\item For $2\le p\le \ell$, each vertex in layer~$p$ has~$p$ neighbours in layer $p-1$.\label{cond:descendants}
	\item For $1\le p\le \ell$, every two vertices in layer~$p$ have at most one common neighbour in layer $p-1$, and at most one common neighbour in layer $p+1$.\label{cond:oneneigh}
\end{enumerate}

\noindent
Note that such a layering does not exist for all graphs.
In particular, if a graph is not bipartite, then it is impossible to partition its vertices according to the four conditions above.
However, we will show that partitioning the independent sets of any graph~$G$ according to their size yields a ``natural'' layering of $\mathcal{I}^1_\mathrm{AR}(G)$.

For convenience, in the remainder of this section we use the same letter to denote an independent set of~$G$ and its associated vertex in $\mathcal{I}^1_\mathrm{AR}(G)$.

\begin{lemma}
\label{lem:tar1}\mbox{}\\*
	For a graph~$G$, for $p=1,\ldots,\alpha(G)$ let~$V_p$ be the collection of independent sets of~$G$ of size~$p$.
	Then $V_1, \ldots, V_{\alpha(G)}$ is a layering of $\mathcal{I}^1_\mathrm{AR}(G)$.
\end{lemma}

\begin{proof}
	Let~$I$ be an independent set of~$G$.
	Since the TAR rule consists of adding or removing a token, the neighbours of~$I$ in $\mathcal{I}^1_\mathrm{AR}(G)$ are independent sets of size $|I|+1$ or $|I|-1$, which implies that both conditions~\ref{cond:indep} and~\ref{cond:neighbours} hold.
	We have that $\{I\setminus\{v\} \mid\linebreak[1] v\in I\}\subseteq V_{|I|-1}$ is the set of all neighbours of~$I$ of size $|I|-1$, thus condition~\ref{cond:descendants} holds.
	Finally, let~$I$ and~$J$ be two independent sets of the same size.
	If they differ by exchanging at least two vertices, then they do not have any common neighbours.
	Suppose they differ by exchanging exactly one vertex.
	Then they have exactly one common neighbour ($I \cap J$) in $V_{|I|-1}$, and at most one common neighbour ($I\cup J$) in $V_{|I|+1}$, thus condition~\ref{cond:oneneigh} holds.
\end{proof}

\noindent\noindent
To conclude the proof that a graph~$G$ can be determined from $\mathcal{I}^1_\mathrm{AR}(G)$, we first prove Lemma~\ref{lem:tar2}, showing that $\mathcal{I}^1_\mathrm{AR}(G)$ admits a unique layering, except for one very specific case.
We then deal with this exceptional case in Lemmas~\ref{lem:cc-decomposition} and~\ref{lem:cycle-reconf}.

\begin{lemma}
\label{lem:tar2}\mbox{}\\*
	Let~$G$ be a graph such that $\mathcal{I}^1_\mathrm{AR}(G)$ has no connected component that is an even cycle.
	Then the layering of $\mathcal{I}^1_\mathrm{AR}(G)$ defined in Lemma~\ref{lem:tar1} is the unique layering of $\mathcal{I}^1_\mathrm{AR}(G)$.
\end{lemma}

\begin{proof}
	Let $V_1, \ldots, V_\ell$ be a layering of $\mathcal{I}^1_\mathrm{AR}(G)$.
	We show by induction on~$p$ that each~$V_p$ contains only independent sets of size~$p$, except for connected components of $\mathcal{I}^1_\mathrm{AR}(G)$ that form an even cycle.

	\textbf{Base case $p=1$.}
	We start by showing that~$V_1$ contains no independent set of size at least~$3$.
	Assuming otherwise, let $I=J\cup\{u,v\}$, $u\ne v$ and $J\ne\varnothing$, be such an independent set.
	By condition~\ref{cond:neighbours}, all neighbours of~$I$ lie in~$V_2$, so in particular $I_1=J\cup\{u\}$ and $I_2=J\cup\{v\}$ lie in~$V_2$.
	By condition~\ref{cond:descendants}, $I_1$ must have another neighbour than~$I$ in~$V_1$, which is either $I_1\setminus\{x\}$ for some $x\in I_1$, or $I_1^+=I_1\cup\{y_1\}$ for some $y_1\notin I$.
	The first case is impossible, since it implies that $I_1$ and $I\setminus\{x\}$ are common neighbours of~$I$ and $I_1\setminus\{x\}$ in~$V_2$, contradicting condition~\ref{cond:oneneigh}.
	Therefore, the neighbours of~$I_1$ in~$V_1$ are~$I$ and~$I_1^+$.
	By the same argument we get that~$I_2$ has neighbours~$I$ and $I_2^+=I_2\cup\{y_2\}$ in~$V_1$ for some $y_2\notin I$.
	Since~$I_1$ and~$I_2$ already have a common neighbour~$I$ in~$V_1$, by conditions~\ref{cond:neighbours} and~\ref{cond:oneneigh} the other common neighbour~$J$ must lie in~$V_3$.
	Next observe that $I_1^+\setminus\{u\}=J\cup\{y_1\}$ and $I_2^+\setminus\{v\}=J\cup\{y_2\}$ both lie in~$V_2$ by condition~\ref{cond:neighbours}.
	So~$J$ has neighbours $I_1,I_2,J\cup\{y_1\},J\cup\{y_2\}$ in~$V_2$.
	By condition~\ref{cond:descendants}, $J$ can only have three neighbours in~$V_2$, so we must have $y_1=y_2$.
	Now we have that all of $I=J\cup\{u,v\}$, $I_1^+=J\cup\{u,y_1\}$, and $I_2^+=J\cup\{v,y_1\}$ are independent sets in~$V_1$.
	This means that also $J\cup\{u,v,y_1\}$ is an independent set, which is a neighbour of all of $I,I_1^+,I_2^+$.
	By condition~\ref{cond:neighbours}, $J\cup\{u,v,y_1\}$ lies in~$V_2$, but then the fact that is has three neighbours in~$V_1$ contradicts condition~\ref{cond:descendants}.

	Next we prove that~$V_2$ does not contain independent sets of size~$3$.
	If~$V_2$ contains an independent set $\{u,v,w\}$, then by the previous paragraph any neighbour of size~$4$ cannot lie in~$V_1$.
	Since it must have its two neighbours in~$V_1$ by condition~\ref{cond:descendants}, these must be (up to symmetry) $\{u,v\}$ and $\{u,w\}$.
	But these two sets have $\{u\}$ and $\{u,v,w\}$ as common neighbours in~$V_2$, contradicting condition~\ref{cond:oneneigh}.

	By the previous paragraph, every independent set of size~$2$ in~$V_1$ can only have neighbours of size~$1$ in~$V_2$, hence has degree~$2$.
	Let~$I$ be an independent set of size~$1$ in~$V_2$.
	Then~$I$ has two neighbours in~$V_1$ by condition~\ref{cond:descendants}, which must have size~$2$.
	If~$I$ had a neighbour~$I'$ in~$V_3$, then~$I'$ would have size~$2$ and three neighbours in~$V_2$ by condition~\ref{cond:descendants}.
	One of those three neighbours would therefore have size~$3$ and lie in~$V_2$, which contradicts the previous paragraph.
	We can conclude that if there is an independent set of size~$2$ in~$V_1$, then it must be adjacent to exactly two sets in~$V_2$, both of which have size~$1$; and every set of size~$1$ in~$V_2$ also has exactly two neighbours, both of which have size~$2$ and lie in~$V_1$.
	In other words, if there is an independent set of size~$2$ in~$V_1$, then it must lie in a component of $\mathcal{I}^1_\mathrm{AR}(G)$ that is an even cycle.

	This proves that for components of $\mathcal{I}^1_\mathrm{AR}(G)$ that do not form an even cycle, $V_1$ contains only independent sets of size~$1$, completing the proof of the base case.

	\textbf{Induction step.} Let $p>1$, and suppose that for all $i<p$ we have that~$V_i$ contains only independent sets of size~$i$.
	Let $I\in V_p$.
	By condition~\ref{cond:descendants}, $I$ has~$p$ neighbours in $V_{p-1}$, each of size $p-1$ by the induction hypothesis.
	Let~$I'$ be one of these neighbours.
	If $p=2$, then~$I'$ has size~$1$, and then~$I$ must have size~$2$, which concludes the induction step.
	Otherwise, $p>2$ and~$I'$ has $p-1$ neighbours in $V_{p-2}$, each of size $p-2$ by the induction hypothesis.
	Since~$I'$ is adjacent only to $|I'|=p-1$ independent sets of size $|I'|-1=p-2$, any other neighbour of~$I'$ must have size $|I'|+1=p$.
	In particular, $I$ must have size~$p$, which concludes the induction step.
\end{proof}

\noindent
It remains to investigate what happens when the hypothesis of Lemma~\ref{lem:tar2} is not satisfied, \textit{i.e.}\ when $\mathcal{I}^1_\mathrm{AR}(G)$ contains an even cycle as a connected component.
We first show that $\mathcal{I}^1_\mathrm{AR}(G)$ is not connected only when~$G$ can be decomposed as the complete join of two smaller graphs whose reconfiguration graphs are precisely the connected components of $\mathcal{I}^1_\mathrm{AR}(G)$ (Lemma~\ref{lem:cc-decomposition}).
Then we show that the only graph whose reconfiguration graph~$\mathcal{I}^1_\mathrm{AR}(G)$ is the even cycle~$C_{2i}$ for some $i\ge2$ is the complement of~$C_i$ (Lemma~\ref{lem:cycle-reconf}).
Combining these results with the observations at the start of this section and Lemma~\ref{lem:tar2} concludes the case $k=1$ for TAR.

\begin{lemma}
\label{lem:cc-decomposition}\mbox{}\\*
	Let~$G$ be a graph and let~$C$ be a connected component of $\mathcal{I}^1_\mathrm{AR}(G)$.
	Then~$G$ is the complete join $G_1\bowtie G_2$ with $C\cong\mathcal{I}^1_\mathrm{AR}(G_1)$ and $\mathcal{I}^1_\mathrm{AR}(G)-C\cong\mathcal{I}^1_\mathrm{AR}(G_2)$.
\end{lemma}

\begin{proof}
	Let $X_1=\{v\in V(G)\mid \{v\}\in V(C)\}$ and $X_2=V(G)\setminus X_1$.
	In~$G$, the graphs~$G_1$ and~$G_2$ induced by~$X_1$ and~$X_2$ must be completely joined, since otherwise there would be an independent set $I=\{v_1,v_2\}$ with $v_1\in X_1,v_2\in X_2$.
	But~$I$ is adjacent to both $\{v_1\}\in C$ and $\{v_2\}\notin C$, contradicting that~$C$ is a component.

	Therefore $G\cong G_1\bowtie G_2$, hence every independent set of~$G$ is an independent set of either~$G_1$ or~$G_2$, which shows that $\mathcal{I}^1_\mathrm{AR}(G_1)\cong C$ and $\mathcal{I}^1_\mathrm{AR}(G_2)\cong \mathcal{I}^1_\mathrm{AR}(G)-C$.
\end{proof}

\begin{lemma}
\label{lem:cycle-reconf}\mbox{}\\*
	Let~$G$ be a graph and $i\ge2$ an integer.
	We have $\mathcal{I}^1_\mathrm{AR}(G)\cong C_{2i}$ if and only if $i\ge4$ and $G\cong \bar{C_i}$.
\end{lemma}

\begin{proof}
	($\Leftarrow$)
	The graph $\mathcal{I}^1_\mathrm{AR}(\bar{C_i})$ has~$2i$ vertices, which correspond to the vertices of~$C_i$ and to the pairs of vertices that are adjacent in~$C_i$.
	(Since~$C_i$, $i\ge4$, contains no triangles, $\bar{C_i}$ has no independent sets of size greater than~$2$.)
	Each independent set of size~$1$ in~$\bar{C_i}$ has exactly two neighbours of size~$2$ (the edges of~$C_i$), so $\mathcal{I}^1_\mathrm{AR}(\bar{C_i})\cong C_{2i}$.

	($\Rightarrow$)
	Let~$G$ be a graph such that $\mathcal{I}^1_\mathrm{AR}(G)\cong C_{2i}$; call the vertices $\{v_1, \ldots, v_{2i}\}$.
	Note that since the degree of every vertex in $\mathcal{I}^1_\mathrm{AR}(G)$ is~$2$, $G$ has no independent set of size~$3$.
	Moreover, the independent sets of size~$1$ also have degree~$2$, meaning that every vertex is adjacent to all but two other vertices.
	It follows that every vertex has degree~$2$ in~$\bar{G}$, and so~$\bar{G}$ is the disjoint union of some cycles, say $C_{i_1}\cup\cdots\cup C_{i_p}$.
	Then~$G$ is the complete join of the complements of these cycles, which makes $\mathcal{I}^1_\mathrm{AR}(G)$ the disjoint union of $C_{2i_1},\ldots,C_{2i_p}$.
	Since we assumed $\mathcal{I}^1_\mathrm{AR}(G)\cong C_{2i}$, we must have $p=1$, $i_p=i$ and $G\cong\bar{C_i}$.
	Finally, the fact that~$G$ has no independent set of size~$3$ forces $i\ge4$.
\end{proof}

\subsection{The case \texorpdfstring{$k\ge2$}{k>=2}}

We end this section with the negative results for independent set reconfiguration when $k\ge2$.
We start with an easy observation.
Choose a graph~$H$ whose maximal independent sets have size at most $k-1$.
For any graph~$G$, let $G'=G\bowtie H$ be the complete join of~$G$ and~$H$.
By construction, every independent set of size at least~$k$ in~$G'$ is completely contained in the copy of~$G$ in~$G'$, hence~$G$ and~$G'$ share the same reconfiguration graphs when at least~$k$ tokens are used, regardless of the paradigm.
In particular, we have the following.

\begin{proposition}\mbox{}\\*
	For any number of tokens $k\ge 2$, there exists an infinite family of graphs that have the same reconfiguration graphs $\IkJ$, $\IkS$ and $\IkAR$.
\end{proposition}

\noindent
Because of this, there is no hope of determining~$G$ from any of its independent set reconfiguration graphs when $k\ge2$.
A natural follow-up question is to ask what happens if we forbid this particular construction, for instance by requiring that every time we have a complete join $G\cong G_1\bowtie G_2$, both~$G_1$ and~$G_2$ must contain independent sets of size~$k$.
In this case, the reconfiguration graph of~$G$ is the disjoint union of those of~$G_1$ and~$G_2$.
It is thus sufficient to only consider connected reconfiguration graphs.
The problem now seems to be a lot harder, and we only solve one case: Token Jumping for $k=2$, for graphs without a universal vertex (a vertex adjacent to all other vertices).

Observing that independent sets of size~$2$ in~$G$ correspond to non-edges of~$G$, we have that if $k=2$, then $\IkJ(G)$ is precisely the line graph of the complement~$\bar{G}$ of~$G$.
For a graph without isolated vertices we have that the graph is connected if and only if its line graph is connected.
So if we assume $\IkJ(G)$ is connected, then so is~$\bar{G}$ (since~$G$ has no universal vertices).
Therefore, using Whitney's Theorem~\cite{whitney1992} again, we deduce the following.

\begin{proposition}
\label{lem:TJ}\mbox{}\\*
	Let~$G$ be a graph without universal vertices such that $\mathcal{I}^2_\mathrm{J}(G)$ is connected.
	Then~$G$ can be determined from $\mathcal{I}^2_\mathrm{J}(G)$, unless $\mathcal{I}^2_\mathrm{J}(G)$ is a triangle~$K_3$ and~$\bar{G}$ is either~$K_3$ or the claw~$K_{1,3}$.
\end{proposition}

\noindent
We know very little about the possibility of determining a graph~$G$ from $\IkS(G)$ or $\IkAR(G)$ if $k\ge2$, or from $\IkJ(G)$ if $k\ge3$.
It seems likely that it is possible if we only consider graphs from specific classes, or if we impose restrictions on these reconfiguration graphs (as we did in Proposition~\ref{lem:TJ}).

\begin{question}\mbox{}\\*
	\textup{(a) Given $k\ge2$, for what classes of graphs $\mcG$ is it possible to reconstruct any $G\in\mcG$ from one of $\IkJ(G)$, $\IkS(G)$ or $\IkAR(G)$?}

	\noindent
	\textup{(b) Given $k\ge2$, for what classes of graphs $\mcG$ is it possible to reconstruct any $G\in\mcG$ from one of $\IkJ(G)$, $\IkS(G)$ or $\IkAR(G)$, if the reconfiguration graph under consideration is connected?}
\end{question}

\paragraph{Acknowledgment}\mbox{}\\*
GB, CB, PH and TP would like to thank Marthe Bonamy for the helpful discussions about the lower bound construction in the Kempe-recolouring cases.

The authors thank two anonymous reviewers for their exceptionally careful reading and for their many suggestions that improved the presentation of the paper.

\printbibliography

\end{document}